\documentclass[11pt]{article}
\usepackage[T1]{fontenc}
\usepackage{lmodern}
\usepackage[a4paper,margin=28mm]{geometry}
\usepackage{amsmath,amssymb,amsthm,mathtools}
\usepackage{microtype}
\usepackage{enumitem}
\usepackage[hidelinks]{hyperref}
\usepackage{bookmark}
\usepackage{needspace}
\usepackage{xcolor}
\DeclareRobustCommand{\blueedit}[1]{{\color{black}#1}}
\newenvironment{bluepassage}{\par\begingroup\color{black}}{\par\endgroup}
\DeclareRobustCommand{\rev}[1]{{#1}}
\newenvironment{revision}{\par\begingroup}{\par\endgroup}
\DeclareRobustCommand{\revred}[1]{{#1}}

\hypersetup{pdftitle={Non-existence of Negative Weight Derivations of the Moduli Algebras of Isolated Weighted Homogeneous Hypersurface Singularities},pdfauthor={Guorui Ma}}
\setlist[enumerate]{label=\textup{(\roman*)},leftmargin=2em,itemsep=3pt}
\numberwithin{equation}{section}
\newtheorem{theorem}{Theorem}[section]
\newtheorem{lemma}[theorem]{Lemma}
\newtheorem{proposition}[theorem]{Proposition}
\newtheorem{corollary}[theorem]{Corollary}
\newtheorem*{maintheorem}{Main Theorem}

\theoremstyle{definition}

\theoremstyle{remark}
\newtheorem{remark}[theorem]{Remark}
\newcommand{\C}{\mathbb C}
\newcommand{\Z}{\mathbb Z}

\newcommand{\m}{\mathfrak m}
\newcommand{\p}{\mathfrak p}
\newcommand{\Oan}{\mathcal O}
\newcommand{\Der}{\operatorname{Der}_{\C}}
\newcommand{\Spec}{\operatorname{Spec}}
\newcommand{\Min}{\operatorname{Min}}
\newcommand{\red}{\mathrm{red}}

\newcommand{\length}{\operatorname{length}}

\title{Non-existence of Negative Weight Derivations of the Moduli Algebras of Isolated Weighted Homogeneous Hypersurface Singularities}
\author{Guorui Ma\\[3pt]
{\small Shanghai Institute for Mathematics and Interdisciplinary Sciences, Shanghai, 200433, China}\\[1pt]
{\small Research institute of Intelligent Complex Systems, Fudan University, Shanghai, 200433, China}\\[1pt]
{\small\href{mailto:mgr18@tsinghua.org.cn}{\texttt{mgr18@tsinghua.org.cn}}}}
\date{}
\makeatletter
\renewcommand{\@maketitle}{%
  \newpage
  \begingroup
  \centering
  {\LARGE\@title\par}
  \vskip 0.8em
  {\normalsize\@author\par}
  \vskip 0.9em
  \endgroup
}
\makeatother
\begin{document}
\maketitle
\begin{abstract}
Mather and Yau proved that the analytic isomorphism class of an isolated complex hypersurface singularity is determined by its dimension and its moduli algebra. Yau's conjecture asserts that the moduli algebra of an isolated weighted homogeneous hypersurface singularity over $\C$ has no nonzero derivation of negative degree. We prove this conjecture.
\end{abstract}
%\tableofcontents

\section{Introduction}

\begin{bluepassage}
Let $f\in\Oan=\C\{x_1,\ldots,x_n\}$ satisfy $f(0)=0$ and have an isolated critical point at the origin. Its Milnor algebra and Tjurina algebra are
\[
M_f=\Oan/(\partial_1f,\ldots,\partial_nf),\qquad
T_f=\Oan/(f,\partial_1f,\ldots,\partial_nf),
\]
where $\partial_i=\partial/\partial x_i$. The Tjurina algebra is also called the moduli algebra. Mather and Yau proved that two isolated complex hypersurface singularities of the same dimension are analytically isomorphic if and only if their moduli algebras are isomorphic \cite{MY82}.

Suppose that $f$ is weighted homogeneous of degree $d$ for positive rational weights $w_1,\ldots,w_n$. Positivity of the weights makes $f$ a polynomial. We scale the weights and $d$ to positive integers and put $W=\max_iw_i$. Euler's identity $df=\sum_iw_ix_i\partial_if$ gives $T_f=M_f$. The weights define a grading on this algebra. We write $\Der(M_f)_k$ for the space of $\C$-linear derivations $\delta$ such that $\delta((M_f)_j)\subseteq(M_f)_{j+k}$ for every $j$, with negative graded components understood to be zero.

Yau conjectured that $\Der(M_f)_k=0$ for every $k<0$. The conjecture arose in the study of derivations of moduli algebras and their relation to quasihomogeneity; see \cite[p.~244]{CXY95} and \cite{XY96}. Chen, Xu and Yau proved the cases of two and three variables \cite[Theorems~2.2 and~2.3]{CXY95}. Chen proved the case of four variables \cite[Theorem~2.1]{Chen95}. Yau and Zuo proved the conjecture under the weight conditions $d\geq2w_1\geq\cdots\geq2w_n>0$ and $w_1\leq2w_n$ \cite[Main Theorem, p.~21]{YZ16}. Chen, Yau and Zuo proved it for weighted homogeneous fewnomial singularities of multiplicity at least five \cite[Main Theorem~C]{CYZ19}. Here a fewnomial defining polynomial has exactly $n$ monomials in $n$ variables, and multiplicity means the least ordinary total degree \cite[Definition~9.1]{CYZ19}. Chen and Yau obtained the sufficient bound $d>W^2-W+1$ \cite[Theorem~1.5]{CY26}.

Yau's conjecture is the Jacobian-algebra case of Halperin's conjecture. The latter asserts that a positively graded Artinian complete intersection over a field of characteristic zero has no nonzero derivation of negative degree.
\end{bluepassage}

\begin{maintheorem}
Let $f\in\C\{x_1,\ldots,x_n\}$ be weighted homogeneous with positive weights and have an isolated critical point at $0$. Then
\begin{align*}
\Der(M_f)_k=0 \qquad\text{for every }k<0.
\end{align*}
\end{maintheorem}

\begin{bluepassage}
Chen, Liu, Yau and Zuo \cite{CLYZ26} also give a proof of the conjecture, over any field of characteristic zero. Their main argument establishes a degree inequality for homogeneous polynomials whose Jacobian ideals are contained in that of $f$. Both proofs use flows of negative-degree derivations and weighted scaling. The proof here uses a different intermediate result concerning radicals of ideals generated by selected partial derivatives.

To state this result, write $S=\C[x_1,\ldots,x_n]$, $f_i=\partial_if$, $J=(f_1,\ldots,f_n)$ and $\m=(x_1,\ldots,x_n)$. Let $S_r$ be the degree-$r$ part of $S$. Weight-preserving coordinate changes and removal of quadratic variables reduce the proof to $S/J=\C$ or to
\(f\in\m^3\), \(d>2W\).
In the latter case, Theorem~\ref{thm:radical} proves that
\[
\sqrt{(f_i:w_i>r)}\cap S_r=\{0\}\qquad(r\in\Z_{>0}).
\]
For $3r<d$, a polynomial $a$ in this intersection would give a deformation $f+ta^2$ with constant Milnor number. The description of the $\mu$-constant stratum in \cite[Theorem~14.8, p.~415]{AGV88} then gives $a^2\in J$. A separate degree argument excludes this membership when $a\ne0$. For $3r\geq d$, the variables of weight greater than $r$ occur at most quadratically, and the radical statement follows by solving their linear gradient equations over the field of fractions of the remaining polynomial ring.

After the quadratic variables have been removed, Lemma~\ref{lem:lift} gives a unique homogeneous polynomial lift of each negative-degree derivation of $S/J$. If such a derivation is nonzero, Proposition~\ref{prop:minimum} produces a nonzero lift $D$ with $D(J)\subseteq J$ and $D^2f=0$. Lemma~\ref{lem:orbit} uses the polynomial orbits of a locally nilpotent derivation and weighted scaling to show that $D$ acts trivially on a suitable reduced quotient. The coefficient of a variable of least weight moved by $D$ then belongs to the intersection excluded by Theorem~\ref{thm:radical}. The orbit lemma applies to finitely generated graded algebras that need not be reduced, and the locally nilpotent derivation in that lemma need not be homogeneous.

Section~\ref{sec:splitting} gives the quadratic reduction and the choice of $D$. Section~\ref{sec:radicals} proves the radical statement. Section~\ref{sec:mainproof} proves the orbit lemma and completes the proof of the Main Theorem. Appendices~\ref{app:local} and~\ref{app:derivations} record the algebraic facts used in the proof.

\paragraph{Formal verification.}
All results proved in this paper have been formally verified in Lean. The standard theorem of Arnold, Gusein-Zade and Varchenko \cite[Theorem~14.8, p.~415]{AGV88} is used as the sole unformalized external mathematical input. All remaining steps in the proof of the Main Theorem have been formally verified. The formalization was completed jointly with Quan Zhang. The code is available at \url{https://github.com/GuoruiMaGRM/Yau-Conjecture}.
\end{bluepassage}

\Needspace{9\baselineskip}
\subsection*{\blueedit{AI Declaration}}
\begin{bluepassage}
ChatGPT 6 suggested applying Varchenko's result on the $\mu$-constant stratum. It also identified the formulation in Arnold, Gusein-Zade and Varchenko \cite[Theorem~14.8, p.~415]{AGV88} as a more explicit reference than the original statement for the result needed in Lemma~\ref{thm:varchenko}. ChatGPT 6 was also used to correct grammar and improve the exposition. The author is responsible for the mathematical content, the references, and the final text.
\end{bluepassage}

\section{\rev{Removal of quadratic variables and the choice of \texorpdfstring{$D$}{D}}}\label{sec:splitting}
\begin{bluepassage}
We use the notation $S$, $\m$, $f_i$, $J$ and $W$ introduced above, and set $A=S/J$. For $a\in\Z$, let $S_a$ be the space spanned by monomials of weight $a$. Then $S_0=\C$ and $S_a=0$ for $a<0$. Unless ordinary degree is specified, degree means weighted degree. Lemma~\ref{lem:analytic} identifies $A$ with the analytic Milnor algebra. For a ring $R$, we write $R_{\red}=R/\sqrt{(0)}$.
\end{bluepassage}

\rev{A coordinate change in this section means a weight-preserving polynomial automorphism of $S$. It is therefore a graded automorphism.} By the chain rule, its invertible Jacobian matrix maps one gradient ideal to the other. It therefore induces \rev{a graded isomorphism} of Milnor algebras.

\begin{lemma}\label{lem:splitting}
\rev{Let $f$ be a weighted homogeneous polynomial with an isolated critical point at $0$. After weight-preserving coordinate changes and removal of the quadratic variables, the graded algebra $A=S/J$ is isomorphic either to $\C$ or to the Milnor algebra of a weighted homogeneous polynomial, again denoted by $f$, satisfying}
\begin{align*}
f\in\m^3,\qquad d>2W,\qquad W=\max_iw_i.
\end{align*}
The remaining polynomial has an isolated critical point.
\end{lemma}

\begin{proof}
First suppose that a variable $u$ has maximal weight $W>d/2$. No monomial of $f$ contains $u$ twice. Write the other variables as $z$. Then
\begin{align*}
f=uL(z)+F(z),\qquad \deg L=d-W.
\end{align*}
We have $L(0)=0$, since the origin is critical. The polynomial $L$ must have a nonzero linear term. Otherwise all derivatives of $f$ vanish on the $u$-axis: the $u$ derivative is $L(0)=0$, and each $z$ derivative is $uL_{z_i}(0)+F_{z_i}(0)=0$. That would contradict isolatedness.

Choose a variable $v$ occurring linearly in $L$. Its weight is $d-W<W$. After a linear change among variables of weight $d-W$, followed by a nonzero scalar change, write $L=v+P(z')$, where $z'$ denotes all other variables. Every nonlinear monomial of weight $d-W$ uses only variables of smaller weight. Thus $P$ does not involve $v$, and the replacement $v'=L$ is a weight-preserving polynomial automorphism. Rename $v'$ as $v$. Now $f=uv+F(v,z')$.

The polynomial $F(v,z')-F(0,z')$ is divisible by $v$. Put
\begin{align*}
K(v,z')=\frac{F(v,z')-F(0,z')}{v},\qquad u'=u+K(v,z').
\end{align*}
The degree of $K$ is $W$. It is independent of $u$. Hence this is another weight-preserving polynomial automorphism, and
\begin{align*}
f=u'v+F(0,z').
\end{align*}
Its gradient ideal contains $u'$ and $v$. Its graded Milnor algebra is exactly the graded Milnor algebra of $F(0,z')$. The critical set is the product of the origin in the $(u',v)$ variables with the critical set of $F(0,z')$. Thus the latter critical point is isolated. Repeat this operation. Each step removes two variables, so it stops. Every remaining weight is at most $d/2$.

If some remaining variables have weight $d/2$, collect them in a column vector $u$ and write the other variables as $z$. The weighted degree condition gives
\begin{align*}
f=\tfrac12u^{\mathsf T}Hu+u^{\mathsf T}a(z)+b(z),
\end{align*}
where $H$ is a constant symmetric matrix. Each entry of $a$ has weight $d/2$. Since every $z$ variable has smaller weight, $a$ has no constant or linear terms. 

The matrix $H$ is invertible. If $0\ne v\in\ker H$, then $(u,z)=(tv,0)$ is critical for every $t$. Indeed, the $u$ derivatives are $Hu+a(0)=0$, and the $z$ derivatives vanish because $a$ has no linear terms and $db(0)=0$. This would again contradict isolatedness.

Set $u'=u+H^{-1}a(z)$. This preserves the weights and gives
\begin{align*}
f=\tfrac12(u')^{\mathsf T}Hu'+b(z)-\tfrac12a(z)^{\mathsf T}H^{-1}a(z).
\end{align*}
The gradient ideal of the first summand is the ideal generated by the entries of $u'$. Removing these variables therefore preserves the graded Milnor algebra and isolatedness.

All remaining weights are strictly less than $d/2$. Hence no monomial of ordinary degree two has weight $d$. There are also no constant or linear terms. The remaining polynomial belongs to $\m^3$ and satisfies $d>2W$. If no variables remain, the quotient algebra is $\C$.
\end{proof}

\rev{We reuse the symbols $x_i,w_i,f,S,J,A$ for the remaining variables, weights, polynomial, and algebras after splitting. From now on, unless otherwise stated, $f\in\m^3$ and $d>2W$. In particular, every $f_i$ has weighted degree at least $d-W>W$.}

Define the graded vector space
\begin{align*}
\mathcal V=\{g\in J:\partial_i g\in J\text{ for every }i\},\qquad
\mathcal V_q=\mathcal V\cap S_q.
\end{align*}

\begin{lemma}\label{lem:potential}
\begin{revision}
For an integer $q$ with $0<q<d$, the map
\begin{align*}
\Der(A)_{q-d}&\longrightarrow\mathcal V_q,\\
\delta&\longmapsto Df,
\end{align*}
is a linear bijection, where $D$ is the unique homogeneous polynomial lift of $\delta$ given by Lemma~\ref{lem:lift}.
\end{revision}
\end{lemma}

\begin{proof}
\rev{Let $\delta\in\Der(A)_{q-d}$ and write its unique lift as $D=\sum_jh_j\partial_j$. Lemma~\ref{lem:lift} gives $D(J)\subseteq J$.} Then $Df=\sum_jh_jf_j\in J$. The product rule gives, for every $i$,
\begin{align*}
\partial_i(Df)=D(f_i)+\sum_j(\partial_i h_j)f_j\in J.
\end{align*}
Thus $Df\in\mathcal V_q$.

Conversely, let $g\in\mathcal V_q$. Since $g\in J$, taking homogeneous components gives $g=\sum_jh_jf_j$ with $h_j\in S_{q-d+w_j}$. Define $D=\sum_jh_j\partial_j$. The same identity gives
\begin{align*}
D(f_i)=\partial_i g-\sum_j(\partial_i h_j)f_j\in J.
\end{align*}
Thus $D$ preserves $J$ and induces \rev{a derivation $\delta\in\Der(A)_{q-d}$.} Lemma~\ref{lem:maximal} excludes nonzero constant coefficients. The other coefficients have positive degree, so all belong to $\m$. They are the unique lifts in Lemma~\ref{lem:lift}. Lemma~\ref{lem:syzygy} proves uniqueness of the expression for $g$, and hence injectivity of the map.
\end{proof}

\begin{lemma}\label{lem:stable}
If $D$ is any polynomial derivation with $D(J)\subseteq J$, then $D(\mathcal V)\subseteq\mathcal V$.
\end{lemma}

\begin{proof}
For $g\in\mathcal V$, we have $Dg\in J$. With $D=\sum_jh_j\partial_j$, differentiation gives
\begin{align*}
\partial_i(Dg)=D(\partial_i g)+\sum_j(\partial_i h_j)(\partial_j g).
\end{align*}
The first term belongs to $J$ because $\partial_i g\in J$ and $D$ preserves $J$. Every term in the sum belongs to $J$ because $\partial_jg\in J$. Hence $Dg\in\mathcal V$.
\end{proof}

\begin{proposition}\label{prop:minimum}
If $A$ has a nonzero homogeneous derivation of negative degree, then there is a nonzero homogeneous polynomial derivation $D$ of degree $-s$, with $s>0$, such that
\begin{align*}
D(J)\subseteq J,\qquad 0\ne g=Df,\qquad Dg=D^2f=0.
\end{align*}
Every coefficient of $D$ belongs to $\m$, and $D$ is locally nilpotent.
\end{proposition}

\begin{proof}
\rev{Let $0\ne\delta\in\Der(A)_{-s_0}$ with $s_0>0$, and let $D_0$ be its lift from Lemma~\ref{lem:lift}. The lift is nonzero, and $D_0f\ne0$ by Lemma~\ref{lem:syzygy}, since its total weighted degree is $d-s_0<d<2(d-W)$. The product rule gives $D_0f\in\mathcal V$. As $J$ has no nonzero element of nonpositive degree, $0<d-s_0<d$. Hence $\mathcal V_q\ne0$ for some integer $q$ with $0<q<d$.}

Choose the least such integer $q$ and a nonzero $g\in\mathcal V_q$. Lemma~\ref{lem:potential} gives a derivation $D$ with $Df=g$ and degree $-s=q-d<0$. Lemma~\ref{lem:stable} gives $Dg\in\mathcal V$. Its degree is $q-s<q$. If that degree is positive, minimality of $q$ implies $Dg=0$. If it is nonpositive, membership in $J$ implies the same conclusion. The remaining assertions follow from Lemma~\ref{lem:lift}.
\end{proof}

\section{\rev{Radicals of ideals generated by selected partial derivatives}}\label{sec:radicals}

\begin{lemma}\label{lem:quadratic}
Let $F(z,Y)$ be a convergent power series germ (or polynomial), where $z=(z_1,\ldots,z_p)$ and $Y=(Y_1,\ldots,Y_q)$, with $q\geq1$. Suppose that $F$ has an isolated critical point at $0$, belongs to $(z,Y)^3$, and is polynomial of ordinary degree at most two in $Y$. Write
\begin{align*}
F(z,Y)=\tfrac12Y^{\mathsf T}H(z)Y+L(z)^{\mathsf T}Y+c(z),\qquad H(z)^{\mathsf T}=H(z).
\end{align*}
Then $\det H(z)$ is not the zero power series (or zero polynomial). If $k(z)$ is a column of $q$ convergent power series (or polynomials) and $k(z)^{\mathsf T}F_Y(z,Y)$ is independent of $Y$, then $k=0$. Here $F_Y$ is the column of the $Y$ derivatives.
\end{lemma}

\begin{proof}
The assumption $F\in(z,Y)^3$ gives $H(0)=0$, $L\in(z)^2$, and $c\in(z)^3$. Put $H_j=(\partial H/\partial z_j)(0)$. Along $z=0$, the derivatives are
\begin{align*}
F_Y(0,Y)=0,\qquad F_{z_j}(0,Y)=\tfrac12Y^{\mathsf T}H_jY.
\end{align*}
\blueedit{If $v^{\mathsf T}H_jv=0$ for every $j$, then $v=0$. Indeed, a nonzero such $v$ would make every point $(z,Y)=(0,tv)$, for sufficiently small $t$, a critical point of $F$, contrary to isolatedness. In particular, the hypotheses exclude $p=0$.}

Let $\mathcal H$ be the vector space spanned by the symmetric matrices $H_1,\ldots,H_p$ over $\C$. We show that it contains an invertible matrix. Choose $M\in\mathcal H$ of maximal rank $e$. If $e<q$, choose $0\ne v\in\ker M$. The symmetric bilinear form with matrix $M$ has an $e$-dimensional subspace $U$ on which it is nondegenerate. This follows by choosing a basis for its nonsingular part; equivalently, a symmetric bilinear form in characteristic different from two can be diagonalized by a change of basis. Since $v$ is in the kernel, $U\cap\C v=0$.

Take any $N\in\mathcal H$. The rank of $M+tN$ is at most $e$ for every $t\in\C$, by maximality of $e$. Its restriction to $U\oplus\C v$ therefore has determinant zero for every $t$. In a basis consisting of a basis of $U$ followed by $v$, that restriction has matrix
\begin{align*}
\begin{pmatrix}
M_U+tN_U & t b\\
tb^{\mathsf T} & t v^{\mathsf T}Nv
\end{pmatrix}.
\end{align*}
The coefficient of $t$ in its determinant is $\det(M_U)v^{\mathsf T}Nv$. Terms that use an off-diagonal entry in the last row must also use one in the last column, so they have order at least two in $t$. Since $M_U$ is invertible, we obtain $v^{\mathsf T}Nv=0$. If $e=0$, the same argument uses the one-dimensional space $\C v$ and gives this equality directly.

This holds for every $N\in\mathcal H$, hence for every $H_j$. It contradicts the absence of a common nonzero zero of the quadratic forms. Thus $e=q$, and $\mathcal H$ contains an invertible matrix.

The linear part of $H(z)$ is $\sum_jz_jH_j$. Its determinant is a nonzero homogeneous polynomial of ordinary degree $q$. Because every entry of $H(0)$ is zero, this polynomial is exactly the ordinary degree-$q$ part of $\det H(z)$. Hence $\det H$ is not identically zero.

Finally, $F_Y=H(z)Y+L(z)$. The coefficient of $Y$ in $k(z)^{\mathsf T}F_Y$ is $k(z)^{\mathsf T}H(z)$. Independence of $Y$ implies $H(z)k(z)=0$. Over the fraction field of $\C\{z\}$ the matrix $H$ is invertible, so $k=0$. For polynomials, the same argument uses $\C(z)$, the fraction field of $\C[z]$.
\end{proof}

\begin{lemma}\label{lem:square}
Assume $f\in S_d\cap\m^3$, $d>2W$, and $f$ has an isolated critical point. If $0\ne a\in S_r$ and $0<3r<d$, then $a^2\notin J$.
\end{lemma}

\begin{proof}
Suppose instead that $a^2\in J$. Write
\begin{align*}
a^2=\sum_i k_if_i,\qquad k_i\in S_{2r-d+w_i}.
\end{align*}
Only indices with $w_i\geq d-2r$ can have nonzero coefficients. Denote all variables with these weights by $Y$, and all other variables by $z$. If there are no $Y$ variables, all $k_i$ vanish, which already contradicts $a\ne0$.

Since $3r<d$, every $Y$ variable has weight at least $d-2r>r$ and also greater than $d/3$. Thus $a$ is independent of $Y$. Further,
\begin{align*}
\deg k_i\leq W+2r-d<r.
\end{align*}
The strict inequality follows from $W<d/2$ and $r<d/3$, which imply $W+r<d$. Therefore every nonzero $k_i$ is also independent of $Y$.

Three $Y$ variables, counted with repetition, have total weight greater than $d$. Hence $f$ has ordinary degree at most two in $Y$. The displayed expression for $a^2$ now has the form
\begin{align*}
k(z)^{\mathsf T}f_Y(z,Y)=a(z)^2.
\end{align*}
Its right side is independent of $Y$. Lemma~\ref{lem:quadratic} forces $k=0$. This gives $a^2=0$ in the polynomial ring, contrary to $a\ne0$.
\end{proof}
%\subsection{A square deformation with constant Milnor number}
For an analytic germ $F$ with isolated critical point, its Milnor number is
\begin{align*}
\mu(F)=\dim_{\C}\Oan/(F_{x_1},\ldots,F_{x_n}).
\end{align*}
A one-parameter analytic deformation of $f$ is a convergent power series $F(x,t)$ with $F(x,0)=f(x)$. We write $f_t(x)=F(x,t)$. It has constant Milnor number at the origin if $0$ is critical for all small $t$, remains isolated, and $\mu(f_t)=\mu(f)$.

\begin{lemma}\label{lem:mu}
Let $f(u,y)$ be an analytic germ with an isolated critical point at $0$ and $f(0)=0$. Let $u=(u_1,\ldots,u_p)$ with $p\geq1$, and $y=(y_1,\ldots,y_q)$. Suppose that $a=a(u)$ belongs to the radical of $(f_{y_1},\ldots,f_{y_q})$. Then $f_t=f+ta^2$ has an isolated critical point at $0$ and $\mu(f_t)=\mu(f)$ \rev{for each $t\in\C$, as a germ at $0$}. For the analytic family near $t=0$, the critical set near $(0,0)$ is, set-theoretically, the zero section $x=0$, and its critical value is zero.
\end{lemma}

\begin{proof}
Write $\Oan=\C\{u,y\}$ and $R=\Oan/(f_{y_1},\ldots,f_{y_q})$. The complete gradient is a regular sequence in $\Oan$. Hence $R$ is Cohen--Macaulay of dimension $p$. The assumption on $a$ means that its image in $R$ is nilpotent. In particular $a(0)=0$, since the ideal of the $y$ derivatives vanishes at the origin.

Use bars for classes in $R$ and set
\begin{align*}
Q_t=(\overline{f_{u_i}}+2t\overline a\,\overline{a_{u_i}}:1\leq i\leq p)\subset R.
\end{align*}
The image of $\overline a$ is zero in $R_{\red}$. Thus $Q_t$ and $Q_0$ have the same image in $R_{\red}$ and the same radical in $R$. The quotient $R/Q_0$ is the Milnor algebra of $f$, so it has finite length. The same is true of $R/Q_t$. Each $Q_t$ is therefore a parameter ideal in the $p$-dimensional ring $R$.

Let $\p$ be a minimal prime with $\dim R/\p=p$. A nilpotent element belongs to every prime ideal, so $\overline a\in\p$. It follows that the images of $Q_t$ and $Q_0$ in $R/\p$ are equal. \rev{The equality on each minimal-prime quotient allows us to apply the associativity formula. Since the ideals are parameter ideals in a Cohen--Macaulay ring, their quotient lengths equal their multiplicities.} Hence,
\begin{align*}
\length_R(R/Q_t)
&=e(Q_t,R)\\
&=\sum_{\substack{\p\in\Min R\\\dim R/\p=p}}
\length_{R_{\p}}(R_{\p})\,e(Q_t(R/\p),R/\p)\\
&=\sum_{\substack{\p\in\Min R\\\dim R/\p=p}}
\length_{R_{\p}}(R_{\p})\,e(Q_0(R/\p),R/\p)\\
&=e(Q_0,R)=\length_R(R/Q_0).
\end{align*}
The $y$ derivatives of $f_t$ are those of $f$, since $a$ depends only on $u$. The $u$ derivatives are $f_{u_i}+2t aa_{u_i}$. Therefore $R/Q_t$ is the Milnor algebra of $f_t$. The equality proves constancy of the Milnor number. Since $a(0)=0$, all first derivatives of $f_t$ vanish at $0$. The finite length just proved gives isolatedness.

We also check the critical set of the family. Work in $\C\{u,y,t\}/(f_y)$. The relation making $a$ nilpotent still holds in this ring. After reduction, the ideals generated by $f_u+2t aa_u$ and by $f_u$ agree. The full gradient ideal of $f_t$ and the extension of the full gradient ideal of $f$ have the same radical in $\C\{u,y,t\}$. The latter radical is $(u,y)$. Hence, set-theoretically, the critical set of the family near $(0,0)$ is exactly $u=y=0$. Finally $f_t(0)=f(0)+ta(0)^2=0$, so the critical value stays zero.
\end{proof}

\begin{revision}
The algebraic argument uses only the first-order consequence stated below
for a linear one-parameter family. We derive it from the description of
the reduced $\mu$-constant locus in \cite{AGV88}, Theorem~14.8. The stratum theorem is due to Varchenko \cite[Theorem~2, p.~30]{Var82}. Saito \cite[Section~3.4, pp.~160--161]{MS91} gives another statement of the same stratum description. The assumptions $f\in\m^3$ and $d>2W$ are not required for the next lemma.

\begin{lemma}\label{thm:varchenko}
Let $f\in S_d$ have an isolated critical point at the origin and satisfy $f(0)=0$.
Let $q\in S_e$, where $e<d$. Suppose that, for every sufficiently small
$t\in\C$, the germ $f_t=f+tq$ has an isolated critical point at the origin,
$f_t(0)=0$, and $\mu(f_t,0)=\mu(f,0)$. Then $q\in J$.
\end{lemma}

\begin{proof}
The assertion is immediate if $q=0$. First suppose that $n=1$.
Write $f=cx^m$ with $c\ne0$ and $m\geq2$. If $q\ne0$, homogeneity
gives $q=bx^\ell$ with $b\ne0$ and $0\leq\ell<m$.
The condition $f_t(0)=0$ excludes $\ell=0$, and criticality at the
origin excludes $\ell=1$. For $2\leq\ell<m$ and $t\ne0$,
\[
 f_t'(x)=x^{\ell-1}\bigl(mc\,x^{m-\ell}+t\ell b\bigr).
\]
The expression in parentheses is a unit at the origin, so
$\mu(f_t,0)=\ell-1<m-1=\mu(f,0)$, a contradiction.
Thus $q=0$ in the one-variable case.

Now assume $n>1$. We use the description of the $\mu$-constant stratum
in \blueedit{\cite[Theorem~14.8, p.~415]{AGV88}}. The construction and
versality of the deformation below are given in
\cite[Section~3.1, pp.~67--68]{AGV88}.
Choose homogeneous monomial representatives
$\phi_0=1,\phi_1,\ldots,\phi_{\mu-1}$ of a basis of $M_f$,
where $\mu=\mu(f,0)$, and put $e_j=\deg\phi_j$.
The miniversal deformation for right equivalence of functions is
\[
 F_s(x)=f(x)+\sum_{j=0}^{\mu-1}s_j\phi_j(x).
\]
Normalize the variable weights to $w_i/d$, so that $f$ has degree one.
The corresponding parameter weights are
\[
 \nu_j=1-\frac{e_j}{d},\qquad
 \nu_j>0\ \Longleftrightarrow\ e_j<d.
\]
Let $\Sigma_\mu$ be the reduced analytic germ of parameters $s$ for
which $F_s$ has a critical point $p$ near the origin satisfying
$F_s(p)=0$ and $\mu(F_s,p)=\mu$.
Theorem~14.8 in \cite{AGV88} gives
\[
 \Sigma_\mu=\{s:s_j=0\text{ whenever }\nu_j>0\}.
\]
Analytic versality gives an analytic map $\gamma(t)$ with $\gamma(0)=0$
and a local biholomorphism
$(x,t)\mapsto(\Phi(x,t),t)$ with $\Phi(x,0)=x$, such that
\[
 f+tq=F_{\gamma(t)}\circ\Phi_t,\qquad \Phi_t(x)=\Phi(x,t).
\]
This is right equivalence of functions. Put $p(t)=\Phi_t(0)$.
The chain rule and the invertibility of the differential of $\Phi_t$
give $dF_{\gamma(t)}(p(t))=0$. Moreover,
\[
 F_{\gamma(t)}(p(t))=f_t(0)=0,\qquad
 \mu(F_{\gamma(t)},p(t))=\mu(f_t,0)=\mu.
\]
Consequently $\gamma(t)\in\Sigma_\mu$, so
$\gamma_j(t)=0$ whenever $e_j<d$.
Differentiating the analytic identity at $t=0$ and writing
$v_i(x)=\partial_t\Phi_i(x,0)$ gives
\[
 q=\sum_i v_i f_i+\sum_{e_j\geq d}\gamma_j'(0)\phi_j.
\]
Take its weighted homogeneous component of degree $e$. The second sum
has no such component because $e<d$. Hence
\[
 q=\sum_i (v_i)_{e-d+w_i}f_i\in J.
\]
Each coefficient $(v_i)_{e-d+w_i}$ is a polynomial because all variable
weights are positive; a component of negative degree is zero.
\end{proof}

\Needspace{9\baselineskip}
\begin{corollary}\label{cor:tangent}
Assume $f\in S_d\cap\m^3$, $d>2W$, and that the origin is an isolated
critical point. Let $a\in S_r$, where $r>0$ and $3r<d$. If
\[
 a\in\sqrt{(f_i:w_i>r)},
\]
then $a^2\in J$.
\end{corollary}

\begin{proof}
If $a=0$, the assertion is immediate. Assume $a\ne0$.
Let $u$ denote the variables of weight at most $r$, and $y$ the remaining
variables. Positivity of the weights implies that $a=a(u)$.
The list $u$ is nonempty, since a nonzero polynomial of positive degree
$r$ cannot be formed using only variables of weight greater than $r$.
If the list $y$ is empty, the ideal in the hypothesis is zero, whose
radical is zero in $S$, contrary to $a\ne0$.

The radical relation remains valid in the analytic local ring, because
the same identity $a^N\in(f_y)$ holds there for some $N\geq1$.
Lemma~\ref{lem:mu} therefore applies to $f+t a^2$.
It gives an isolated critical point at the origin, constant local
Milnor number, and critical value zero for all sufficiently small $t$.
Since $a^2\in S_{2r}$ and $2r<d$, Lemma~\ref{thm:varchenko} gives
$a^2\in J$. The stronger bound $3r<d$ is retained because it is also
needed in Lemma~\ref{lem:square}.
\end{proof}
\end{revision}

%\subsection{The radical statement in every degree}

For a positive integer $r$, define
\begin{align*}
I_r=(f_i:w_i>r)\subset S.
\end{align*}
The ideal generated by an empty list is the zero ideal.

\begin{theorem}\label{thm:radical}
Let $f\in S_d\cap\m^3$ have an isolated critical point, and assume $d>2W$. Then
\begin{align*}
\sqrt{I_r}\cap S_r=\{0\}\qquad\text{for every integer }r>0.
\end{align*}
\end{theorem}

\begin{proof}
Suppose that $0\ne a\in\sqrt{I_r}\cap S_r$. Let $u$ be the list of variables of weight at most $r$, and let $y$ be the list of variables of weight greater than $r$. Since $a$ has degree $r$, it only depends on $u$. If there are no $u$ variables, then $S_r=0$, a contradiction. If there are no $y$ variables, then $I_r=0$, whose radical is zero in $S$, again 
a contradiction. We can therefore assume both lists are nonempty. We have $I_r=(f_y)$.

First suppose that $3r<d$. Corollary~\ref{cor:tangent} gives
$a^2\in J$. This contradicts Lemma~\ref{lem:square}.

It remains to treat $3r\geq d$. \rev{Every $y$ variable has weight strictly greater than $r\geq d/3$.} A product of three $y$ variables has weight greater than $d$. Thus $f$ has ordinary degree at most two in $y$. Write
\begin{align*}
f(u,y)=\tfrac12y^{\mathsf T}H(u)y+L(u)^{\mathsf T}y+c(u).
\end{align*}
By Lemma~\ref{lem:quadratic}, $\det H(u)$ is nonzero. Let $K=\C(u)$, the fraction field of the polynomial ring in the $u$ variables. Over $K$ the matrix $H$ is invertible, so the equations $f_y=Hy+L$ are equivalent to $y=-H^{-1}L$. It follows that
\begin{align*}
K[y]/(f_y)\simeq K.
\end{align*}
In particular, $(f_y)$ is a proper ideal of $K[y]$.

On the other hand, $a\in\sqrt{I_r}$ means that $a^N\in(f_y)$ for some $N\geq1$. Since $a$ is a nonzero polynomial in $u$, it becomes a unit in $K$. After extension to $K[y]$, the ideal $(f_y)$ would contain a unit. This contradicts its being proper. Both ranges of $r$ are excluded.
\end{proof}

\section{Proof of the main theorem}\label{sec:mainproof}

A locally nilpotent derivation $\delta$ of a $\C$-algebra $R$ gives automorphisms
\begin{align*}
\exp(t\delta)(a)=\sum_{\nu\geq0}\frac{t^\nu}{\nu!}\delta^\nu(a)\qquad(a\in R).
\end{align*}
The sum is finite for each $a$. The product rule implies that this map is an algebra homomorphism. Its inverse is $\exp(-t\delta)$, and the maps satisfy the group law for $(\C,+)$. They give the associated additive-group action.

\begin{bluepassage}
We use the Nullstellensatz in the following form: an element of a finitely generated $\C$-algebra vanishes at every complex point if and only if it is nilpotent. For a complex point $p$ and an element $a\in R$, write $a(p)$ for its value at $p$.
\end{bluepassage}

\begin{lemma}\label{lem:orbit}
Let $R=\bigoplus_{a\geq0}R_a$ be a finitely generated graded $\C$-algebra with $R_0=\C$. Suppose that homogeneous elements $z_1,\ldots,z_N$ generate it, with positive integer degrees $w_j\leq W$. Let $q_1,\ldots,q_m$ be homogeneous elements of degrees $d_i>W$ such that $R/(q_1,\ldots,q_m)$ is finite dimensional. If $\delta$ is a locally nilpotent derivation of $R$ and
\begin{align*}
\delta^2q_i=0\qquad(1\leq i\leq m),
\end{align*}
then the derivation induced by $\delta$ on $R_{\red}$ is zero. The derivation $\delta$ is not required to be homogeneous.
\end{lemma}

\begin{proof}
First, $\delta$ preserves the nilradical. For a nilpotent element $a$, each automorphism $\exp(t\delta)$ sends $a$ to a nilpotent element. Thus, in $R_{\red}$, the polynomial
\begin{align*}
\sum_{\nu\geq0}\frac{t^\nu}{\nu!}[\delta^\nu a]
\end{align*}
is zero for every $t\in\C$. Evaluate it at more distinct scalars than its degree. The invertibility of the resulting Vandermonde matrix shows that each coefficient is zero. In particular $\delta a$ is nilpotent. Hence $\delta$ induces a derivation on $R_{\red}$.

Assume that this induced derivation is nonzero. Since the classes of the $z_j$ generate $R_{\red}$, at least one $\delta z_j$ is nonzero there. The Nullstellensatz gives a complex point $p$ at which this element does not vanish. The orbit through $p$ has coordinate polynomials
\begin{align*}
P_j(t)=\bigl(\exp(t\delta)z_j\bigr)(p).
\end{align*}
At least one $P_j$ is nonconstant because $P_j'(0)=(\delta z_j)(p)$ is nonzero for some $j$.

For every nonconstant $P_j$, let $N_j\geq1$ be its ordinary degree in $t$. Define
\begin{align*}
\alpha=\min_{P_j\text{ nonconstant}}\frac{w_j}{N_j}=\frac{a}{b},
\end{align*}
where $a,b$ are positive integers. We have $0<\alpha\leq W$. Introduce a new parameter $\zeta$ and set
\begin{align*}
\gamma_j(\zeta)=\zeta^{bw_j}P_j(\zeta^{-a}).
\end{align*}
Each $\gamma_j$ is a polynomial in $\zeta$. Indeed, if a term of $P_j$ has degree $k\leq N_j$, its exponent after substitution is $bw_j-ak\geq bw_j-aN_j\geq0$. For a constant $P_j$, the exponent is $bw_j>0$. A zero polynomial causes no difficulty.

Choose an index attaining the minimum defining $\alpha$. The leading term of its $P_j$ gives exponent zero, and all lower terms give positive exponents. Its leading coefficient is nonzero. Therefore $\gamma_j(0)\ne0$ for this index.

The grading defines the scaling action $z_j\mapsto\lambda^{w_j}z_j$. For $\zeta\ne0$, the tuple $\gamma(\zeta)$ is obtained by first taking the additive orbit at $t=\zeta^{-a}$, and then scaling with $\lambda=\zeta^b$. It is therefore a point of $\Spec R$. Every defining polynomial relation vanishes on $\gamma(\zeta)$ for nonzero $\zeta$. Since each coordinate is polynomial, the relation also vanishes at $\zeta=0$. Thus $\gamma(0)$ is a point of $\Spec R$.

Because $\delta^2q_i=0$, the orbit of $q_i$ is affine in $t$:
\begin{align*}
\bigl(\exp(t\delta)q_i\bigr)(p)=q_i(p)+t(\delta q_i)(p).
\end{align*}
Homogeneity now gives, for $\zeta\ne0$,
\begin{align*}
q_i(\gamma(\zeta))
=\zeta^{bd_i}q_i(p)+\zeta^{bd_i-a}(\delta q_i)(p).
\end{align*}
Both exponents are positive. For the second, use $d_i>W\geq\alpha=a/b$. Hence $q_i(\gamma(0))=0$ for every $i$.

The only common zero of the $q_i$ is the point where all $z_j$ vanish. To check this, note that the quotient by the $q_i$ is finite dimensional and graded. It has zero homogeneous components in all sufficiently large degrees. A sufficiently high power of each positive-degree $z_j$ is therefore zero in that quotient. Thus every common zero of the $q_i$ has all $z_j=0$.

This contradicts the nonzero coordinate of $\gamma(0)$ found above. The induced derivation on $R_{\red}$ must be zero.
\end{proof}

\begin{remark}
Zero elements among the $q_i$ may be omitted. The argument does not require $R$ to be reduced, a domain, or Cohen--Macaulay. It also does not require the number of $q_i$ to equal $\dim R$.
\end{remark}

%\subsection{The induced action on an ideal of partial derivatives}

\begin{lemma}\label{lem:quotient}
Let $D=\sum_jh_j\partial_j$ be homogeneous of degree $-s<0$. Assume that $D(J)\subseteq J$ and $D^2f=0$. Fix a positive integer $r$ such that $h_j=0$ whenever $w_j\leq r$. Then $D$ preserves $I_r$, and its induced derivation on $(S/I_r)_{\red}$ is zero.
\end{lemma}

\begin{proof}
For each $i$, the polynomial $D(f_i)$ belongs to $J$ and has degree $d-w_i-s$. Taking homogeneous coefficients gives
\begin{align*}
D(f_i)=\sum_j c_{ij}f_j,\qquad c_{ij}\in S_{w_j-w_i-s}.
\end{align*}
If $w_i>r$ and $c_{ij}\ne0$, then $w_j-w_i-s\geq0$, so $w_j\geq w_i+s>r$. Every term belongs to $I_r$. This proves $D(I_r)\subseteq I_r$.

Put $g=Df$ and $g_i=\partial_i g$. By the product rule and the hypothesis on the coefficients $h_j$,
\begin{align*}
g_i=D(f_i)+\sum_{w_j>r}(\partial_i h_j)f_j.
\end{align*}
Therefore $D(f_i)\equiv g_i\pmod{I_r}$ for every $i$. When $w_i>r$, the already proved stability of $I_r$ also gives $g_i\in I_r$.

The assumption $D^2f=0$ says $Dg=0$. Differentiate this equality with respect to $x_i$:
\begin{align*}
0=\partial_i(Dg)=D(g_i)+\sum_{w_j>r}(\partial_i h_j)g_j.
\end{align*}
For every index in the sum, $g_j\in I_r$. Consequently $D(g_i)\in I_r$ for all $i$. Apply $D$ to $D(f_i)-g_i\in I_r$. Since $D(I_r)\subseteq I_r$, we obtain
\begin{align*}
D^2(f_i)-D(g_i)\in I_r,\qquad\text{hence }D^2(f_i)\in I_r.
\end{align*}

Now let $R=S/I_r$. It is graded, has $R_0=\C$, and is generated by the classes of the variables, of degrees at most $W$. The negative degree of $D$ makes it locally nilpotent, by the degree argument in Lemma~\ref{lem:lift}. Its induced derivation on $R$ is also locally nilpotent. For the indices with $w_i\leq r$, take $q_i$ to be the image of $f_i$ in $R$. Each nonzero such element is homogeneous of degree
\begin{align*}
d-w_i\geq d-W>W.
\end{align*}
The remaining partial derivatives are already zero in $R$. Thus
\begin{align*}
R/(q_i:w_i\leq r)=S/J,
\end{align*}
which is finite dimensional. The calculation above gives $D^2q_i=0$. Lemma~\ref{lem:orbit} applies and shows that the induced derivation on $R_{\red}$ is zero.
\end{proof}

%\subsection{Completion of the proof}

\begin{proof}[Proof of the Main Theorem]
Multiply the positive rational weights by a common positive integer. This changes no sign of a derivation degree. The function is a polynomial because all weights are positive. Lemma~\ref{lem:analytic} identifies its analytic Milnor algebra with the polynomial one. Euler's identity identifies this algebra with the moduli algebra.

By Lemma~\ref{lem:splitting}, a graded algebra isomorphism carries degree-$k$ derivations to degree-$k$ derivations by conjugation. Thus it suffices to prove the assertion for the algebra after splitting. If that algebra is $\C$, every $\C$-linear derivation is zero. Otherwise we have $f\in\m^3$ and $d>2W$.

Suppose, for a contradiction, that there is a nonzero homogeneous derivation of negative degree. Proposition~\ref{prop:minimum} gives a nonzero polynomial derivation $D$ of degree $-s<0$ such that
\begin{align*}
D(J)\subseteq J,\qquad D^2f=0,\qquad h_i=D(x_i)\in\m\cap S_{w_i-s}.
\end{align*}
Choose an index $k$ for which $h_k\ne0$ and $w_k$ is as small as possible. Set
\begin{align*}
h=h_k,\qquad r=w_k-s.
\end{align*}
The polynomial $h$ is nonzero, homogeneous, and belongs to $\m$. Hence its degree $r$ is a positive integer. Since $s>0$, we also have $r<w_k$.

If $w_i\leq r$, then $w_i<w_k$. By the choice of $k$, $h_i=0$. The hypotheses of Lemma~\ref{lem:quotient} are therefore satisfied. It follows that the induced derivation on $(S/I_r)_{\red}$ is zero.

In particular, the class of $D(x_k)=h$ is zero in $(S/I_r)_{\red}$. By the definition of reduction, this means that the image of $h$ in $S/I_r$ is nilpotent. Equivalently, $h^N\in I_r$ for some $N\geq1$, so
\begin{align*}
0\ne h\in\sqrt{I_r}\cap S_r.
\end{align*}
This contradicts Theorem~\ref{thm:radical}. Thus no nonzero homogeneous negative derivation exists. Undoing the graded isomorphisms from the splitting gives the stated conclusion for the original Milnor and moduli algebras.
\end{proof}

\begin{remark}
For a finite-dimensional graded algebra, every derivation is a finite sum of homogeneous derivations. Indeed, for homogeneous $a\in A_j$, project $\delta(a)$ onto $A_{j+k}$ to define its degree-$k$ component. Projecting the product rule shows that each component is a derivation. Only finitely many components occur because only finitely many $A_j$ are nonzero. The Main Theorem therefore says that the entire derivation algebra has no negative graded part.
\end{remark}

\appendix
\section{Local algebras}\label{app:local}
\begin{bluepassage}
We recall the commutative algebra used in Lemma~\ref{lem:mu}. All rings are commutative with identity. We write $\length_R M$ for module length and $\Min R$ for the set of minimal prime ideals. Over a local $\C$-algebra with residue field $\C$, finite length equals dimension over $\C$.

The rings $S_{\m}$ and $\Oan=\C\{x_1,\ldots,x_n\}$ are regular local rings of dimension $n$, hence Cohen--Macaulay. In a Cohen--Macaulay local ring, every system of parameters is a regular sequence in any order. A quotient by $c$ members of a regular sequence is Cohen--Macaulay of dimension lower by $c$; see \cite[Tags~\href{https://stacks.math.columbia.edu/tag/00NQ}{00NQ} and \href{https://stacks.math.columbia.edu/tag/02JN}{02JN}]{Stacks}.
\end{bluepassage}

Since $J$ has finite colength, $f_1,\ldots,f_n$ is a regular sequence in $S_{\m}$ and in $\Oan$. It is also regular in $S$. To see the last assertion, let $H$ be any homogeneous ideal generated by an initial part of this sequence. The map $S/H\to(S/H)_{\m}$ is injective. Indeed, if $ub=0$ in $S/H$ for a polynomial $u$ with $u(0)\ne0$ and a nonzero element $b$, the component of least degree of $ub$ is $u(0)$ times the component of least degree of $b$. It cannot vanish. A failure of regularity in $S$ would therefore remain a failure after localization, which is impossible. The same argument applies after any reordering.

\begin{bluepassage}
Let $Q$ be a parameter ideal in an $r$-dimensional Cohen--Macaulay local ring $R$. Its Hilbert--Samuel multiplicity satisfies
\[
\length_R(R/Q)=e(Q,R).
\]
This follows from \cite[Tag~\href{https://stacks.math.columbia.edu/tag/0AZZ}{0AZZ}]{Stacks}, since the higher Koszul homology groups of a regular sequence vanish. The associativity formula gives
\[
e(Q,R)=\sum_{\substack{\p\in\Min R\\\dim R/\p=r}}
\length_{R_{\p}}(R_{\p})\,e(Q(R/\p),R/\p);
\]
see \cite[Tag~\href{https://stacks.math.columbia.edu/tag/0AZX}{0AZX}]{Stacks}. Here $Q(R/\p)$ denotes the image of $Q$ in $R/\p$. In Lemma~\ref{lem:mu}, these formulas are applied to a quotient of $\Oan$ by part of the gradient sequence.
\end{bluepassage}

\begin{lemma}\label{lem:analytic}
If the weighted homogeneous polynomial $f$ has an isolated critical point at $0$, then $\sqrt J=\m$. Moreover, the natural map
\begin{align*}
S/J\longrightarrow\Oan/J\Oan
\end{align*}
is an isomorphism. In particular, $A$ is a finite-dimensional local algebra with maximal ideal $\m/J$.
\end{lemma}

\begin{proof}
For $\lambda\in\C^*$, define $\lambda\cdot x=(\lambda^{w_1}x_1,\ldots,\lambda^{w_n}x_n)$. Homogeneity gives
\begin{align*}
f_i(\lambda\cdot x)=\lambda^{d-w_i}f_i(x).
\end{align*}
If a nonzero point $p$ were critical, then every $\lambda\cdot p$ would be critical. As $\lambda$ tends to zero, these nonzero points tend to zero. This contradicts the isolated critical point assumption. Thus the common zero set of $J$ in $\C^n$ is $\{0\}$. \blueedit{The Hilbert Nullstellensatz gives $\sqrt J=\m$.}

For each $i$, choose $N_i$ with $x_i^{N_i}\in J$. If $N=1+\sum_i(N_i-1)$, every monomial of ordinary degree at least $N$ is divisible by some $x_i^{N_i}$. Hence $\m^N\subset J$. Modulo $J\Oan$, a convergent power series equals its Taylor polynomial of ordinary degree less than $N$. This proves surjectivity of the displayed map.

To prove injectivity, suppose that $p\in S$ and $p=\sum_i b_if_i$ with $b_i\in\Oan$. Expand each $b_i$ into its weighted homogeneous components. Each such component is a polynomial, since all weights are positive. Taking the component of degree $a$ gives
\begin{align*}
p_a=\sum_i (b_i)_{a-d+w_i}f_i\in J.
\end{align*}
The polynomial $p$ has finitely many components. Their sum belongs to $J$. Thus the map is injective. The quotient is finite dimensional because $\m^N\subset J$. It is local because its only maximal ideal is $\m/J$.
\end{proof}

We will also use the argument in the last paragraph for an ideal generated by any list of homogeneous polynomials. In particular, analytic membership of a homogeneous polynomial in such an ideal implies polynomial membership.

\section{Locally nilpotent derivations}\label{app:derivations}

Throughout this appendix, $f\in S_d\cap\m^3$ has an isolated critical
point at the origin and $d>2W$, where $W=\max_i w_i$, as obtained
after the reduction in Lemma~\ref{lem:splitting}.

\begin{lemma}\label{lem:syzygy}
Suppose that $q<2(d-W)$, that $a_i\in S_{q-d+w_i}$ for each $i$, and that $\sum_i a_if_i=0$. Then all coefficients $a_i$ are zero. Consequently, an expression $g=\sum_i a_if_i$ with $g$ homogeneous of degree less than $d$ has unique homogeneous coefficients of degrees $\deg g-d+w_i$.
\end{lemma}

\begin{proof}
Suppose some $a_i$ is nonzero. Let $j$ be the largest such index. If $j=1$, then $a_1f_1=0$ contradicts regularity. If $j>1$, the relation implies
\begin{align*}
a_jf_j\in(f_1,\ldots,f_{j-1}).
\end{align*}
Since $f_j$ is a nonzero divisor modulo this ideal, $a_j$ belongs to $(f_1,\ldots,f_{j-1})$. This homogeneous ideal has no nonzero element of degree less than $d-W$. Hence $\deg a_j\geq d-W$. Since $\deg f_j\geq d-W$, the total degree is at least $2(d-W)$, a contradiction.

For uniqueness, subtract two homogeneous expressions for $g$. Their difference is a relation of total degree $\deg g<d<2(d-W)$. The first assertion applies.
\end{proof}

\begin{lemma}\label{lem:maximal}
Every derivation of a finite-dimensional local $\C$-algebra preserves its maximal ideal.
\end{lemma}

\begin{proof}
\blueedit{Let $B$ be such an algebra and let $\mathfrak n$ be its maximal ideal. Since $B$ is Artinian, $\mathfrak n$ is nilpotent.}

Take $0\ne a\in\mathfrak n$, and choose the least integer $N$ with $a^N=0$. If $\delta(a)\notin\mathfrak n$, it is a unit. Applying $\delta$ gives
\begin{align*}
0=\delta(a^N)=Na^{N-1}\delta(a).
\end{align*}
Multiplication by the inverse of $N\delta(a)$ would give $a^{N-1}=0$, contrary to the choice of $N$. Therefore $\delta(a)\in\mathfrak n$. The assertion for $a=0$ is immediate.
\end{proof}

\begin{lemma}\label{lem:lift}
\rev{Let $s$ be a positive integer.} Every $\delta\in\Der(A)_{-s}$ has a unique lift
\begin{align*}
D=\sum_{i=1}^n h_i\frac{\partial}{\partial x_i},\qquad
h_i\in\m\cap S_{w_i-s},\qquad D(J)\subseteq J.
\end{align*}
The derivation $D$ is locally nilpotent: for every $p\in S$, some power $D^N(p)$ is zero.
\end{lemma}

\begin{proof}
The element $\delta([x_i])$ has degree $w_i-s<W$, where brackets denote the class in $A$. The ideal $J$ has no nonzero element in that degree. Hence $\delta([x_i])$ has a unique homogeneous representative $h_i\in S_{w_i-s}$. If this degree is negative, the representative is zero. If it is zero, Lemma~\ref{lem:maximal} shows that the representative is zero as well. In all cases $h_i\in\m$.

Define $D$ by the displayed formula, and let $\pi:S\to A$ be the quotient map. The maps $\pi D$ and $\delta\pi$ agree on the generators. The product rule then shows that they agree on all of $S$. If $p\in J$, this gives $\pi(Dp)=\delta(\pi p)=0$. Hence $D(J)\subseteq J$, and $D$ induces $\delta$ on the quotient. Uniqueness follows from uniqueness of each $h_i$.

For a homogeneous polynomial $p$ of degree $a$, the polynomial $D^N(p)$ has degree $a-Ns$, or is zero. It is zero when $Ns>a$. Every polynomial is a finite sum of homogeneous polynomials, so one sufficiently large $N$ works for the whole polynomial.
\end{proof}

\section*{Acknowledgement}
\begin{bluepassage}
I thank Professor Stephen Yau and Huaiqing Zuo for introducing this problem to me early in my PhD studies and for guiding me into this field. I thank Quan Zhang for completing the Lean formalization jointly with me, and Liangliang Shi and Rui Tao for helpful discussions. The use of ChatGPT 6 is described in the AI Declaration.
\end{bluepassage}

\end{document}